\documentclass[11pt]{oupau-ppn}
\usepackage{amsmath,amsfonts,amssymb,amsxtra,setspace,xspace,color}
\usepackage[colorlinks=true]{hyperref}
\hypersetup{urlcolor=blue, citecolor=red, linkcolor=blue}
\newtheorem{thm}{Theorem} 
\newtheorem{cor}[thm]{Corollary}
\newtheorem{lem}[thm]{Lemma}
\newtheorem{prop}[thm]{Proposition}
\newcommand{\R}{{\mathbb R}}

\newcommand{\be}[1]{\begin{equation}\label{#1}}
\newcommand{\ee}{\end{equation}}
\renewcommand{\(}{\left(}
\renewcommand{\)}{\right)}
\newcommand{\ird}[1]{\int_{\R^d}{#1}\,dx}
\newcommand{\irdmu}[1]{\int_{\R^d}{#1}\,d\mu}
\newcommand{\nrm}[2]{\|{#1}\|_{\mathrm L^{#2}(\R^d)}}
\newcommand{\nrmu}[2]{\|{#1}\|_{\mathrm L^{#2}(\R^d\!,\;d\mu)}}

\newcommand\email[1]{\emph{E-mail:} \textsf{#1}}

\usepackage{natbib}
\setcitestyle{authoryear,open={[},close={]}}

\begin{document}
\title{Stability results for logarithmic Sobolev and Gagliardo-Nirenberg inequalities}
\shorttitle{Stability in logarithmic Sobolev and related interpolation inequalities}
\author{Jean Dolbeault\affil{1} and Giuseppe Toscani\affil{2}}
\abbrevauthor{Dolbeault, J., and Toscani, G.}
\headabbrevauthor{J.~Dolbeault and G.~Toscani}
\address{
\affilnum{1} Ceremade, UMR CNRS nr.~7534, Universit\'e Paris-Dauphine, Place de Lattre de Tassigny, 75775 Paris Cedex~16, France, \email{dolbeaul@ceremade.dauphine.fr},\\
\affilnum{2} Department of Mathematics, University of Pavia, via Ferrata 1, 27100 Pavia, Italy,\newline\email{giuseppe.toscani@unipv.it}.}

\begin{abstract}
\hspace*{-2.8pt}\parbox{17cm}{This paper is devoted to improvements of functional inequalities based on scalings and written in terms of relative entropies. When scales are taken into account and second moments fixed accordingly, deficit functionals provide ex\-pli\-cit stability measurements, \emph{i.e.}, bound with explicit constants distances to the manifold of optimal functions. Various re\-sults are obtained for the Gaussian logarithmic Sobolev inequality and its Euclidean counterpart, for the Gaussian generalized Poincar\'e inequalities and for the Gagliardo-Nirenberg inequalities. As a consequence, faster convergence rates in diffusion equations (fast diffusion, Ornstein-Uhlenbeck and porous medium equations) are obtained.}\end{abstract}
\maketitle

\par\bigskip

\section{Introduction}\label{Sec:Intro}

Several papers have recently been devoted to \emph{improvements of the logarithmic Sobolev inequality}. \cite{2014arXiv1403.5855L} use the Stein discrepancy. Closer to our approach is \cite{2014arXiv1408.2115B,2014arXiv1410.6922F}, who exploit the difference between the inequality of \cite[Inequality~(2.3)]{MR0109101} and the logarithmic Sobolev inequality to get a correction term in terms of the \emph{Fisher information} functional. What we do here first is to emphasize the role of scalings and prefer to rely on \cite{MR479373} for a \emph{scale invariant form of the logarithmic Sobolev inequality} on the Euclidean space. We also make the choice to get a remainder term that involves the entropy functional and is very appropriate for stability issues. This allows us to deduce striking results in terms of rates of convergence for the Ornstein-Uhlenbeck equation. Writing the improvement in terms of the entropy has several advantages: contraints on the second moment are made clear, improvements can be extended to all generalized Poincar\'e inequalities for Gaussian measures, which interpolate between the Poincar\'e inequality and the logarithmic Sobolev inequality, and stability results with fully explicit constants can be stated: see for instance Corollary~\ref{Cor:Interpolation}, with an explicit bound of the distance to the manifold of all Gaussian functions given in terms of the so-called \emph{deficit functional}. This is, for the logarithmic Sobolev inequality, the exact analogue of the result of \cite{MR1124290} for Sobolev's inequality.

However, putting the emphasis on scalings has other advantages, as the method easily extends to a nonlinear setting. We are henceforth in a position to get \emph{improved entropy -- entropy production inequalities associated with fast diffusion flows} based on the scale invariant forms of the associated Gagliardo-Nirenberg inequalities, which cover a well-known family of inequalities that contain the logarithmic Sobolev inequality, and Sobolev's inequality as an endpoint. This is not a complete surprise because such improvements were known from \cite{MR3103175} using detailed properties of the fast diffusion equation. By writing the entropy -- entropy production inequality in terms of the relative entropy functional and a generalized Fisher information, we deduce from the scaling properties of Gagliardo-Nirenberg inequalities a correction term involving the square of the relative entropy, and this is much simpler than using the properties of the nonlinear flow. The method also works in the porous medium case, which is new, provides clear evidences on the role of the second moment, and finally explains the fast rates of convergence in relative entropy that can be observed in the initial regime, away from Barenblatt equilibrium or self-similar states.

The reader interested in further considerations on \emph{improvements of the logarithmic Sobolev inequality} is invited to refer to \cite{2014arXiv1403.5855L} and \cite{2014arXiv1408.2115B,2014arXiv1410.6922F} for probabilistic point of view and a measure of the defect in terms of Wasserstein's distance, and to \cite{MR1849347} for earlier results. Much more can also be found in \cite{MR3155209}. Not all Gagliardo-Nirenberg-Sobolev inequalities are covered by our remarks and we shall refer to \cite{MR3177378} and references therein for the spectral point of view and its applications to the Schr\"odinger operator. The logarithmic Sobolev inequality in scale invariant form is equivalent to the \emph{Gaussian isoperimetric inequality}: a study of the corresponding deficit can be found in \cite{Mossel-Neeman}. In the perspective of information theory, we refer to \cite{MR3034582,2014arXiv1410.2722T} for a recent account on a concavity property of \emph{entropy powers} that involves the isoperimetric inequality. It is not possible to quote all earlier related contributions but at least let us point two of them: the correction to the logarithmic Sobolev inequality by an entropy term involving the Wiener transform in \cite[Theorem~6]{MR1132315}, and the \emph{HWI inequality} by \cite{MR1760620}.

\emph{Gagliardo-Nirenberg inequalities} (see \cite{MR0102740,MR0109940}) have been related with fast diffusion or porous media equations in the framework of the so-called entropy methods by \cite{MR1940370}. Also see the papers by \cite{MR1777035,MR1842429,MR1986060} for closely related issues. The message is simple: optimal rates of convergence measured in relative entropy are equivalent to best constant in the inequalities written in entropy -- entropy production form. Later improvements have been obtained on asymptotic rates of convergence by \cite{BBDGV,BDGV,1004,DKM}. A key observation of \cite{1004} is the fact that optimizing a relative entropy with respect to scales determines the second moment. This observation was then exploited by \cite{MR3103175} to get a first \emph{explicit} improvement in the framework of Gagliardo-Nirenberg inequalities. Notice that many papers on improved interpolation inequalities use the estimate of \cite{MR1124290} with the major drawback that the value of the constant is not known. As a consequence the improved inequality, faster convergence rates for the solution to the fast diffusion equation were obtained and a new phenomenon, a delay, was shown by \cite{2014arXiv1408.6781D}. Inspired by \cite{MR1768665}, \cite{MR3200617} studied the $p$-th R\'enyi entropy and observed that the corresponding isoperimetric inequality is a Gagliardo-Nirenberg inequality in scale invariant form. Various consequences for the solutions to the evolution equations have been drawn in \cite{carrillo2014renyi} and \cite{DTS}, which are strongly related with the present paper but can all be summarized in a simple sentence: scales are important and a better adjustment than the one given by the asymptotic regime gives sharper estimates. The counterpart in the present paper is that taking into account the scale invariant form of the inequality automatically improves on the inequality obtained by a simple entropy -- entropy production method. Let us give some explanations.

\medskip At a formal level, the strategy of our paper goes as follows. Let us consider a \emph{generalized entropy} functional~$\mathcal E$, which is assumed to be nonnegative, and a \emph{generalized Fisher information} functional $\mathcal I$. We further assume that they are related by a functional inequality of the form
\[
\mathcal I-\lambda\,\mathcal E\ge0\,.
\]
We denote by $\lambda$ the optimal proportionality constant. If the inequality is not in scale invariant form, we will prove in various cases that there exists a convex function $\varphi$, leaving from $\varphi(0)=0$ with $\varphi'(0)=\lambda$ such that $\mathcal I\ge\varphi(\mathcal E)$. Hence we have found an \emph{improved functional inequality} in the sense that
\[
\mathcal I-\lambda\,\mathcal E\ge\varphi(\mathcal E)-\lambda\,\mathcal E=\psi(\mathcal E)
\]
where $\psi(\mathcal E)$ is nonnegative and can be used to measure the distance to the optimal functions. This is a stability result. The left hand side, which is called the deficit functional in the literature, is now controlled from below by a nonlinear function of the entropy functional. A precise distance can be obtained by the Pinsker-Csisz\'ar-Kullback inequality, which is no more than a Taylor expansion at order two, and some generalizations.The key observation is that the optimization under scaling (in the Euclidean space) amounts to adjust the second moment (in the Euclidean space but also in spaces with finite measure, like the Gaussian measure, after some changes of variables).

At this point it is worth to emphasize the difference in our approach compared to the one of \cite{2014arXiv1408.2115B} for the logarithmic Sobolev inequality. What the authors do is that they write the improved inequality as $\varphi^{-1}(\mathcal I)\ge\mathcal E$ and deduce that
\[
\mathcal I-\lambda\,\mathcal E\ge\mathcal I-\lambda\,\varphi^{-1}(\mathcal I)
\]
where the right hand side is again nonnegative because $\varphi^{-1}$ is concave and $\lambda\,(\varphi^{-1})'(0)=1$. This is of course a stronger form of the inequality, as it controls the distance to the manifold of optimal functions in a stronger norm, for instance. However, it is to a large extend useless for the applications that are presented in this paper, as the estimate in terms of the entropy is what matters, for instance, for application in evolution equations.

We shall apply our strategy to the logarithmic Sobolev inequality in Section~\ref{Sec:logSob}, to the generalized Poincar\'e inequalities for Gaussian measures in Section~\ref{Sec:Beckner} and to some Gagliardo-Nirenberg inequalities in Section~\ref{Sec:GN}. Each of these inequalities can be established by the \emph{entropy -- entropy production} method. By considering the Ornstein-Uhlenbeck equation in the first two cases, and the fast diffusion / porous medium equation in the third case, it turns out that $\frac{d\mathcal E}{dt}=-\,\mathcal I$ and
\[
-\,\frac d{dt}\(\mathcal I-\lambda\,\mathcal E\)=\mathcal R\ge0\,.
\]
Hence, if $\lim_{t\to\infty}\(\mathcal I-\lambda\,\mathcal E\)=0$, this shows with no additional assumption that $\int_0^\infty\mathcal R[v(t,\cdot)]\,dt$ is a measure of the distance to the optimal functions. \emph{Improved functional inequalities} follow by ODE techniques if one is able to relate $\mathcal R$ with $\mathcal E$ and $\mathcal I$. This is the method which has been implemented for instance in \cite{MR2152502,pre05312043,MR3103175} and it is well adapted when the diffusion equation can be seen as the gradient flow of $\mathcal E$ with respect to a distance. Typical distances are the Wasserstein distance for the logarithmic Sobolev inequality or the Gagliardo-Nirenberg inequalities, and \emph{ad hoc} distances in case of the generalized Poincar\'e inequalities. See \cite{MR1617171,MR1842429,MR2448650,MR3023408} for more details on gradient flow issues. Improvements can also be obtained when $\frac{d\mathcal E}{dt}$ differs from $-\,\mathcal I$: we refer to \cite{MR2381156,DEKL} for interpolation inequalities on compact manifolds, or to \cite{MR3227280} for improvements of Sobolev's inequality based on the Hardy-Littlewood-Sobolev functional. This makes the link with the famous improvement obtained by \cite{MR1124290}, and also \cite{MR2538501}, but so far no \emph{entropy -- entropy production} method has been able to provide an improvement in such a critical case. For completeness, let us mention that other methods can be used to obtain improved inequalities, which are based on variational methods like in \cite{MR1124290}, on symmetrization techniques like in \cite{MR2538501} or on spectral methods connected with heat flows like in \cite{MR2375056}. Here we shall simply rely on convexity estimates and the interplay of entropy -- entropy production inequalities with their scale invariant counterparts.

A very interesting feature of \emph{improved functional inequalities} in the framework of the \emph{entropy -- entropy production} method is that the entropy decays faster than expected by considering only the asymptotic regime. In that sense, the improved inequality capture an initial rate of convergence which is faster than the asymptotic one. This has already been observed for fast diffusion equations in \cite{MR3103175} with a phenomenon of delay that has been studied in \cite{2014arXiv1408.6781D} and by \cite{carrillo2014renyi}, by resorting to the concept of R\'enyi entropy. A remarkable fact is that the inequality is improved by choosing a scale (in practice by imposing a constraint on the second moment) without requesting anything on the first moment, again something that clearly distinguishes the improvements obtained here from what can be guessed by looking at the asymptotic problem as $t\to\infty$. Details and statements on these consequences for diffusion equations have been collected in Section~\ref{Sec:Diffusion}.

\section{Stability results for the logarithmic Sobolev inequality}\label{Sec:logSob}

Let $d\mu=\mu\,dx$ be the normalized Gaussian measure, with $\mu(x)=(2\,\pi)^{-d/2}\,e^{-|x|^2/2}$, on the Euclidean space $\R^d$ with $d\ge1$. The \emph{Gaussian logarithmic Sobolev inequality} reads
\be{Ineq:LogSobGaussian}
\irdmu{|\nabla u|^2}\ge\frac12\,\irdmu{|u|^2\,\log|u|^2}
\ee
for any function $u\in\mathrm H^1(\R^d,d\mu)$ such that $\irdmu{|u|^2}=1$. This inequality is equivalent to the \emph{Euclidean logarithmic Sobolev inequality in scale invariant form}
\be{Ineq:LogSobEuclideanWeissler}
\frac d2\,\log\(\frac2{\pi\,d\,e}\ird{|\nabla w|^2}\)\ge\ird{|w|^2\,\log|w|^2}
\ee
that can be found in \cite[Theorem~2]{MR479373} in the framework of scalings, but is also the one that can be found in \cite[Inequality~(2.3)]{MR0109101} or in \cite[Inequality~(26)]{MR1132315}. See \cite{2014arXiv1408.2115B,2014arXiv1410.6922F} and \cite{Tos2014} for more comments. The equivalence of~\eqref{Ineq:LogSobGaussian} and~\eqref{Ineq:LogSobEuclideanWeissler} is well known but involves some scalings and we will give a short proof below for completeness. Next, let us consider the function
\be{Eqn:varphi}
\varphi(t):=\frac d4\,\left[\exp\(\frac{2\,t}d\)-1-\frac{2\,t}d\right]\quad\forall\,t\in\R\,.
\ee
Our first result is an improvement of \eqref{Ineq:LogSobGaussian}, based on the comparison of~\eqref{Ineq:LogSobGaussian} with~\eqref{Ineq:LogSobEuclideanWeissler}, which combines ideas of \cite{MR2294794} and \cite{2014arXiv1410.6922F}. It goes as follows.
\begin{prop}\label{Thm:LogSob} With $\varphi$ defined by \eqref{Eqn:varphi}, we have
\begin{multline}\label{Ineq:LogSobGaussianImproved}
\irdmu{|\nabla u|^2}-\frac12\,\irdmu{|u|^2\,\log|u|^2}\ge\varphi\(\irdmu{|u|^2\,\log|u|^2}\)\\
\forall\,u\in\mathrm H^1(\R^d,d\mu)\quad\mbox{such that}\quad\irdmu{|u|^2}=1\quad\mbox{and}\quad\irdmu{|x|^2\,|u|^2}=d\,.
\end{multline}
\end{prop}
Inequality~\eqref{Ineq:LogSobGaussianImproved} is an improvement of \eqref{Ineq:LogSobGaussian} because $\varphi(t)\ge\frac{t^2}{2\,d}$ for any $t\in\R$ and, by the Pinsker-Csisz\'ar-Kullback inequality,
\[
\irdmu{|u|^2\,\log|u|^2}\ge\frac14\,\(\irdmu{\Big|\,|u|^2-1\,\Big|}\)^2\quad\forall\,u\in\mathrm L^2(\R^d,d\mu)\quad\mbox{such that}\quad\nrmu u2=1\,.
\]
See \cite{MR0213190,Csiszar67,Kullback67} for a proof of this inequality.
\begin{proof} To emphasize the role of scalings, let us give a proof of Proposition~\ref{Thm:LogSob}, which follows the strategy of~\cite[Proposition~2, p.~694]{MR2294794}.

As a preliminary step, we recover the scale invariant, Euclidean, version of the logarithmic Sobolev inequality from~\eqref{Ineq:LogSobGaussian}. Let $v:=u\,\sqrt\mu$. We observe that $\ird{|v|^2}=1$ and $\ird{|x|^2\,|v|^2}=d$. With one integration by parts, we get that 
\be{Ineq:LogSobGaussianEuclidean}
\ird{|\nabla v|^2}\ge\frac12\,\ird{|v|^2\,\log|v|^2}+\frac d4\,\log(2\,\pi\,e^2)
\ee
which is the standard \emph{Euclidean logarithmic Sobolev inequality} established in~\cite{Gross75} (also see \cite{Federbush} for an earlier related result). This inequality is not invariant under scaling. By considering $w$ such that $v(x)=\lambda^{d/2}\,w(\lambda\,x)$, we get that
\[
\lambda^2\ird{|\nabla w|^2}-\frac d2\,\log\lambda\ge\frac12\,\ird{|w|^2\,\log|w|^2}+\frac d4\,\log(2\,\pi\,e^2)\,.
\]
holds for any $v\in\mathrm H^1(\R^d,d\mu)$ such that $\nrm v2=1$. An optimization on the scaling parameter shows that $4\,\lambda^2\,\ird{|\nabla w|^2}=d$ and establishes the scale invariant form of the logarithmic Sobolev inequality,
\be{Ineq:LogSobWeissler}
\frac d2\,\log\(\frac2{\pi\,d\,e}\ird{|\nabla w|^2}\)\ge\ird{|w|^2\,\log|w|^2}\quad\forall\,w\in\mathrm H^1(\R^d,dx)\quad\mbox{such that}\quad\nrm w2=1\,,
\ee
which is equivalent to \eqref{Ineq:LogSobEuclideanWeissler}. This inequality can also be written as
\[
\ird{|\nabla w|^2}\ge\frac12\,\pi\,d\,e\,\exp\(\frac 2d\ird{|w|^2\,\log|w|^2}\)\,.
\]
If we redefine $u$ such that $w=u\,\sqrt\mu$ and assume that $\nrm w2=1$, $\ird{|x|^2\,w}=d$, we have shown that
\be{Eqn:LogSobGaussianSimple}
\irdmu{|\nabla u|^2}\ge\frac d4\left[\exp\(\frac 2d\irdmu{|u|^2\,\log|u|^2}\)-1\right]\,.
\ee
Inequality~\eqref{Ineq:LogSobGaussianImproved} follows by substracting $\frac12\,\irdmu{|u|^2\,\log|u|^2}$ from both sides of the inequality, which is more or less the idea that has been exploited by \cite{2014arXiv1410.6922F}. \end{proof}

Consider a nonnegative function $f\in\mathrm L^1_2(\R^d):=\left\{g\in\mathrm L^1(\R^d)\,:\,\ird{|x|^2\,g}<\infty\right\}$ and, assuming that $\ird f>0$, define
\be{Def:M-theta}
M_f:=\ird f\,,\quad\theta_{\!f}:=\frac 1d\,\frac{\ird{|x|^2\,f}}{M_f}\,.
\ee
Let us define the Gaussian function
\[
\mu_f(x):=\frac{M_f}{(2\,\pi\,\theta_{\!f})^{d/2}}\,e^{-\frac{|x|^2}{2\,\theta_{\!f}}}\quad\forall\,x\in\R^d\,.
\]
We shall denote by $\mathrm L^1_2(\R^d)$ the space of integrable functions on $\R^d$ with finite second moment.
\begin{lem}\label{Lem:Euclidean} Assume that $f$ is a nontrivial, nonnegative function in $f\in\mathrm L^1_2(\R^d)$ such that $\nabla\sqrt f\in\mathrm L^2(\R^d)$. With $\theta_{\!f}$, $\mu_f$ and $\varphi$ defined by \eqref{Eqn:varphi} and \eqref{Def:M-theta}, we have
\be{Ineq:ImprovevdEuclidean}
\frac{\theta_{\!f}}2\ird{\frac{|\nabla f|^2}f}-\ird{f\,\log f}-\frac d2\,\log\(2\,\pi\,e^2\,\theta_{\!f}\)\,\ird f\ge2\,\varphi\left[\,\ird{f\,\log\(\frac f{\mu_f}\)}\right]\,.
\ee
\end{lem}
\begin{proof} Let $v$ be such that $\lambda^d\,f(\lambda\,x)=|u(x)|^2\,\mu(x)$ with $\lambda^2=\theta_{\!f}$, $\mu(x)=(2\,\pi)^{-d/2}\,e^{-|x|^2/2}$, and apply Proposition~\ref{Thm:LogSob}.\end{proof}

The Gaussian function $\mu_f$ is the minimizer of the \emph{relative entropy}
\[
\mathsf e[f|\mu]:=\ird{\left[f\,\log\(\frac f\mu\)-(f-\mu)\right]}
\]
w.r.t.~all Gaussian functions in
\[
\mathcal M:=\big\{\mu(x)=\frac M{(2\,\pi\,\theta)^{d/2}}\,:\,e^{-\frac{|x|^2}{2\,\theta}}\,,\;M>0\,,\;\theta>0\big\}\,,
\]
that is, we have the identity
\[
\ird{f\,\log\(\frac f{\mu_f}\)}=\mathsf e[f|\mu_f]=\min\left\{\mathsf e[f|\mu]\,:\,\mu\in\mathcal M\right\}\,.
\]
Also notice that $\mu_f$ is the minimizer of the \emph{relative Fisher information} w.r.t.~all Gaussian functions of mass $M_f$:
\[
\ird{\big|\,\nabla\sqrt{f/\mu_f}\,\big|^2}=\min\left\{\ird{\big|\,\nabla\sqrt{f/\mu}\,\big|^2}\,:\,\mu(x)=\frac{M_f}{(2\,\pi\,\theta)^{d/2}}\,e^{-\frac{|x|^2}{2\,\theta}}\,,\theta>0\right\}\,.
\]

Recall that by the Pinsker-Csisz\'ar-Kullback inequality, the r.h.s.~in~\eqref{Ineq:ImprovevdEuclidean} provides an explicit stability result in $\mathrm L^1_+(\R^d,dx)$ that can be written as
\[
\mathsf e[f|\mu_f]\ge\frac1{4\,M_f}\,\nrm{f-\mu_f}1^2\quad\forall\,f\in\mathrm L^1_+(\R^d,dx)\,.
\]
Combined with the observation that $\varphi$ is nondecreasing and $\varphi(t)\ge\frac{t^2}{2\,d}$ for any $t\in\R$, we have shown the following global stability result.
\begin{cor}\label{Cor:Interpolation} Assume that $f$ is a nontrivial, nonnegative function in $f\in\mathrm L^1_2(\R^d)$ such that $\nabla\sqrt f\in\mathrm L^2(\R^d)$. With $\theta_{\!f}$, $\mu_f$ and $\varphi$ defined by \eqref{Eqn:varphi} and \eqref{Def:M-theta}, we have
\begin{multline*}
\frac{\theta_{\!f}}2\ird{\frac{|\nabla f|^2}f}-\ird{f\,\log f}-\frac d2\,\log\(2\,\pi\,e^2\,\theta_{\!f}\)\,\ird f\\
\ge2\,\min_{\mu\in\mathcal M}\varphi\big(\mathsf e[f|\mu]\big)=2\,\varphi\big(\mathsf e[f|\mu_f]\big)\ge\frac{\nrm{f-\mu_f}1^4}{16\,M_f^2}\,.
\end{multline*}
\end{cor}

\section{An improved version of the generalized Poincar\'e inequalities for Gaussian measures}\label{Sec:Beckner}

We consider the inequalities introduced by W.~Beckner in \cite[theorem~1]{MR954373}. If $\mu(x)=(2\,\pi)^{-d/2}\,e^{-|x|^2/2}$, then for any $p\in[1,2)$ we have
\be{Ineq:Beckner}
\nrmu u2^2-\nrmu up^2\le(2-p)\,\nrmu{\nabla u}2^2\quad\forall\,u\in\mathrm H^1(\R^d,d\mu)\,.
\ee
These inequalities interpolate between the Poincar\'e inequality ($p=1$ case) and the logarithmic Sobolev inequality, which is achieved by dividing both sides of the inequality by $(2-p)$ and passing to the limit as $p\to2$. Some improvements were obtained already obtained in \cite{MR2152502,MR2375056,MR2201954}. What we gain here is that the improvement takes place also in the limit case as $p\to2$ and is consistent with the results of Proposition~\ref{Thm:LogSob}.

Let us define
\[
\varphi_p(x):=\frac d4\left[(1-x)^{-\frac{2\,p}{d\,(2-p)}}-1\right]\quad\forall\,x\in[0,1]\,.
\]
\begin{cor}\label{Cor:Beckner} Assume that $u\in\mathrm H^1(\R^d,d\mu)$ is such that $\irdmu{|u|^2\,|x|^2}=d\,\nrmu u2^2$. With the above notation, for any $p\in[1,2)$ we have
\be{Ineq:BecknerImproved}
\irdmu{|\nabla u|^2}\ge\nrmu u2^2\,\varphi_p\(\frac{\nrmu u2^2-\nrmu up^2}{\nrmu u2^2}\)\,.
\ee
\end{cor}
By homogeneity we can assume that $\nrmu u2=1$. The reader is invited to check that
\[
\lim_{p\to2}\varphi_p\(1-\nrmu up^2\)=\frac d4\(e^{\frac 2d\mathsf E[u]}-1\)\quad\mbox{where}\quad\mathsf E[u]:=\irdmu{\frac{|u|^2}{\nrmu u2^2}\,\log\(\frac{|u|^2}{\nrmu u2^2}\)}\,.
\]
The proof of Corollary~\ref{Cor:Beckner} is a straightforward consequence of~\eqref{Eqn:LogSobGaussianSimple} and of the following estimate.
\begin{lem}\label{Lem:Equivalence} For any $p\in[1,2)$ and any function $u\in\mathrm L^p\cap\mathrm L^2(\R^d,d\mu)$, we have
\[
\frac{\nrmu u2^2}{\nrmu up^2}\le \exp\(\tfrac{2-p}p\,\mathsf E[u]\)
\]
and, as a consequence, for any $u\in\mathrm L^p\cap\mathrm L^2(\R^d,d\mu)$ such that $\nrmu u2=1$, we obtain
\be{Ineq:equiv}
\nrmu u2^2-\nrmu up^2\le\frac{2-p}p\irdmu{|u|^2\,\log|u|^2}\,.
\ee
\end{lem}
\begin{proof} The proof relies on an idea that can be found in \cite{MR1796718} and goes as follows. Let us consider the function
\[
k(s):=s\,\log\(\irdmu{u^\frac2s}\)\,.
\]
Derivatives are such that
\begin{eqnarray*}
&&\frac12\,k'(s)=\log\(\irdmu{u^\frac2s}\)-\frac1s\,\frac{\irdmu{u^\frac2s\,\log u}}{\irdmu{u^\frac2s}}\,,\\
&&\frac{s^3}4\(\irdmu{u^\frac2s}\)^2\,k''(s)=\irdmu{u^\frac2s}\irdmu{u^\frac2s\,|\log u|^2}-\(\irdmu{\log u\,u^\frac2s}\)^2\,,\\
\end{eqnarray*}
hence proving that $k$ is convex by the Cauchy-Schwarz inequality. As a consequence we get that
\[
k'(1)\le\frac{k(s)-k(1)}{s-1}\quad\forall\,s>1\,.
\]
Applied with $s=2/p$, this proves that
\[
-\irdmu{|u|^2\,\log\(\frac{|u|^2}{\nrmu u2^2}\)}\le\frac p{2-p}\,\nrmu u2^2\,\log\(\frac{\nrmu up^2}{\nrmu u2^2}\),
\]
from which we deduce the second inequality in \eqref{Ineq:equiv} after observing that $-\log x\ge 1-x$.\end{proof}

The result of Corollary~\ref{Cor:Beckner} deserves a comment. As $x\to0$, $\varphi_p(x)\sim\frac p2\,\frac x{2-p}$, so that we do not recover the optimal constant in~\eqref{Ineq:Beckner} in the asymptotic regime corresponding to $\nrmu up/\nrmu u2\to1$, that is when $u$ approaches a constant, because of the factor $\frac p2$. On the other hand, \eqref{Ineq:BecknerImproved} is a strict improvement compared to~\eqref{Ineq:Beckner} as soon as $\nrmu up^2/\nrmu u2^2<x_\star(p)$ where $x_\star(p)$ is the unique solution to $\varphi_p(x)=\frac x{2-p}$ in $(0,1)$. Let $\Phi_p$ be the function defined by
\be{Phi_p}
\Phi_p(x)=\varphi_p(x)\quad\mbox{if}\quad x\in(0,x_\star(p))\,,\quad\Phi_p(x)=\frac x{2-p}\quad\mbox{if}\quad x\in[x_\star(p),1]\,.
\ee
Collecting these estimates of~\eqref{Ineq:Beckner} and~\eqref{Ineq:BecknerImproved}, we can write that
\[
\irdmu{|\nabla u|^2}\ge\nrmu u2^2\,\Phi_p\(\frac{\nrmu u2^2-\nrmu up^2}{\nrmu u2^2}\)
\]
for any function $u\in\mathrm H^1(\R^d,d\mu)$ such that $\irdmu{|u|^2\,|x|^2}=d\,\nrmu u2^2$. This is an improvement with respect to~\eqref{Ineq:Beckner} because $\Phi_p(x)\ge\frac x{2-p}$, with a strict inequality if $x<x_\star(p)$.

The right hand side in~\eqref{Ineq:BecknerImproved} controls the distance to the constants. Indeed, using for instance H\"older's estimates, it is easy to check that
\[
\nrmu u2^2-\nrmu up^2\ge\nrmu u2^2-\nrmu u1^2=\irdmu{|u-\overline u|^2}
\]
with $\overline u=\irdmu{|u|}$. Sharper estimates based for instance on variants of the Pinsker-Csisz\'ar-Kullback inequality can be found in \cite{MR1951784,BDIK}.

\section{Stability results for some Gagliardo-Nirenberg inequalities}\label{Sec:GN}

\subsection{A first case: \texorpdfstring{$q>1$}{q>1}}\label{Sec:GN1}
We study the case of Gagliardo-Nirenberg inequalities
\be{GN1}
\nrm{\nabla w}2^\vartheta\,\nrm w{q+1}^{1-\vartheta}\ge\mathsf C_{\rm{GN}}\,\nrm w{2q}
\ee
with $\vartheta=\frac dq\,\frac{q-1}{d+2-q\,(d-2)}$. The value of the optimal constant has been established in \cite{MR1940370} (also see \cite{Gunson91} for an earlier but partial contribution).

Let us start with some elementary observations on convexity. Consider two positive constants $a$ and $b$. Let us define 
\[
\zeta=\frac b{a+b}\,,\quad\kappa=\(\frac ab\)^\zeta+\(\frac ba\)^{1-\zeta}=\frac{a+b}{a^{1-\zeta}\,b^\zeta}\,.
\]
Next let us take three positive numbers, $A$, $B$, and $C$ such that $A^\zeta\,B^{1-\zeta}\ge C$ and consider the function 
\[
h(\lambda)=\lambda^a\,A+\lambda^{-b}\,B-\kappa\,C\,.
\]
The function $h$ reaches its minimum at $\lambda=\lambda_*:=\(\frac{b\,B}{a\,A}\)^\frac1{a+b}$ and it is straightforward to check that
\[
h(1)\ge\inf_{\lambda>0}h(\lambda)=h(\lambda_*)=\kappa\(A^\zeta\,B^{1-\zeta}-C\)\,.
\]
This computation determines the choice of $\kappa$. Using the assumption $A^\zeta\,B^{1-\zeta}\ge C$, we get the estimate
\[
A+B-\kappa\,C\ge C^\frac1\zeta\,B^{1-\frac1\zeta}+B-\kappa\,C=\varphi(B_*-B)
\]
where
\[
B_*:=C\(\frac{1-\zeta}\zeta\)^\zeta
\]
and
\be{Eqn:VarphiFD}
\varphi(s):=C^\frac1\zeta\,\left[(B_*-s)^{1-\frac1\zeta}-B_*^{1-\frac1\zeta}\right]-s\,.
\ee
Indeed, $\zeta=\frac b{a+b}$ leads to the identity
 \[
 \kappa\,C = B_* + C^{\frac 1\zeta}\,B_*^{1-\frac 1\zeta}\,.
 \]
Note that $\varphi$ is a nonnegative strictly convex function such that $\varphi(0)=0$ and $\varphi''(s)>\varphi''(0)=\frac{1-\zeta}{\zeta^2}\,C^\frac1\zeta\,B_*^{-1-\frac1\zeta}$ for any $s\in(0,B_*)$.

We apply these preliminary computations with
\begin{eqnarray*}
&&a=\frac dq-(d-2)\,,\quad b=d\,\frac{q-1}{2\,q}\\
&&A=\frac14\,(q^2-1)\ird{|\nabla w|^2}\,,\quad B=\beta\ird{|w|^{q+1}}\,,\quad\beta=\frac{2\,q}{q-1}-d\\
&&C=(\tfrac14\,(q^2-1))^\zeta\,\beta^{1-\zeta}\,\(\mathsf C_{\rm{GN}}\,\nrm w{2q}\)^\alpha\,,\quad \alpha=q+1-\zeta\,(q-1)
\end{eqnarray*}
for any $q\in\big(1,\frac d{d-2}\big)$. With $\mathcal K:=(\frac14\,(q^2-1))^\zeta\,\beta^{1-\zeta}\,\kappa$, the functional
\[
\mathsf J[w]:=\frac14\,(q^2-1)\ird{|\nabla w|^2}+\beta\ird{|w|^{q+1}}-\mathcal K\,\mathsf C_{\rm{GN}}^\alpha\(\ird{|w|^{2q}}\)^\frac\alpha{2q}
\]
is nonnegative and achieves its minimum at $w_*(x)=(1+|x|^2)^\frac1{1-q}$. Hence we have that
\[
\mathsf J[w]\ge\mathsf J[w_*]=0
\]
and this inequality is equivalent to \eqref{GN1}, after an optimization under scaling. Notice that $\vartheta=2\,\zeta/\alpha$.
\begin{thm}\label{Thm:GN1} With the above notations and $\varphi$ given by~\eqref{Eqn:VarphiFD}, we have
\be{Ineq:GNImproved1}
\mathsf J[w]\ge\varphi\left[\beta\(\ird{|w_*|^{q+1}}-\ird{|w|^{q+1}}\)\right]
\ee
for any $w\in\mathrm L^{q+1}(\R^d)$ such that $\ird{|\nabla w|^2}<\infty$ and $\ird{|w|^{2q}\,|x|^2}=\ird{w_*^{2q}\,|x|^2}$. \end{thm}
\begin{proof}
The reader is invited to check that, with the above notations,
\[
B_*-B=\beta\(\ird{|w_*|^{q+1}}-\ird{|w|^{q+1}}\)\,.
\]
\end{proof}

As a last remark in this section, let us observe that the logarithmic Sobolev inequality appears as a limit case of the entropy -- entropy production inequality, and that~\eqref{Ineq:LogSobEuclideanWeissler} is also obtained by taking the limit as $q\to1$ in Gagliardo-Nirenberg inequalities in~\eqref{GN1}: see \cite{MR1940370} for details. Also, when $d\ge2$, the convexity of $\varphi$ is lost as $q\to\frac d{d-2}$, which corresponds to Sobolev's inequality. This shows the consistancy of our method.

\subsection{A second case: \texorpdfstring{$q<1$}{q<1}}\label{Sec:GN2}

Now we study the case of Gagliardo-Nirenberg inequalities
\be{GN2}
\nrm{\nabla w}2^\vartheta\,\nrm w{2q}^{1-\vartheta}\ge\mathsf C_{\rm{GN}}\,\nrm w{q+1}
\ee
with $q=\frac 1{2\,p-1}<1$, $\vartheta=\frac d{1+q}\,\frac{1-q}{d-q\,(d-2)}$ and we denote by $\nrm w{2q}$ the quantity $\(\ird{|w|^{2q}}\)^\frac1{2q}$ for any $q\in(0,1)$, even for $q<1/2$ (in that case, it is only a semi-norm). 

Our elementary estimates have to be adapted. Consider two positive constants $a$ and $b$, with $a>b$. Let us define 
\[
\eta=\frac b{a-b}\,,\quad\kappa=\(\frac ba\)^{\eta}-\(\frac ba\)^{1+\eta}=\frac{a-b}{b^{-\eta}\,a^{1+\eta}}\,.
\]
Next let us take three positive numbers, $A$, $B$, and $C$ such that $A^{-\eta}\,B^{1+\eta}\le C$ and consider the function 
\[
h(\lambda)=\lambda^a\,A-\lambda^b\,B+\kappa\,C\,.
\]
The function $h$ reaches its minimum at $\lambda=\lambda_*:=\(\frac{b\,B}{a\,A}\)^\frac1{a-b}$ and it is straightforward to check that
\[
h(\lambda)\ge h(\lambda_*)=\kappa\(C-A^{-\eta}\,B^{1+\eta}\)\,.
\]
Using the assumption $A^{-\eta}\,B^{1+\eta}\le C$, we get the estimate
\[
A-B+\kappa\,C\ge C^{-\frac1\eta}\,B^{1+\frac1\eta}-B+\kappa\,C=\varphi(B-B_*)
\]
where
\[
B_*:=C\(\frac\eta{1+\eta}\)^\eta
\]
and
\be{Eqn:VarphiPM}
\varphi(s)=C^{-\frac1\eta}\,\left[(B_*+s)^{1+\frac1\eta}-B_*^{1+\frac1\eta}\right]-s
\ee
is a nonnegative strictly convex function such that $\varphi(0)=0$ and $\varphi''(s)>\varphi''(0)>0$ for any $s>0$.

We apply these preliminary computations with
\begin{eqnarray*}
&&a=\frac dq-(d-2)\,,\quad b=d\,\frac{1-q}{2\,q}\\
&&A=\frac14\,(q^2-1)\ird{|\nabla w|^2}\,,\quad B=\beta\ird{|w|^{q+1}}\,,\quad\beta=\frac{2\,q}{1-q}+d\\
&&C=(\tfrac14\,(q^2-1))^{-\eta}\,\beta^{1+\eta}\,\(\mathsf C_{\rm{GN}}\)^{-(q+1)(1+\eta)}\,\nrm w{2q}^\alpha\,,\quad \alpha=q+1+\eta\,(q-1)
\end{eqnarray*}

for any $q\in(0,1)$. With $\mathcal K:=(\tfrac14\,(q^2-1))^{-\eta}\,\beta^{1+\eta}\,\kappa$, the functional
\[
\mathsf J[w]:=\frac14\,(q^2-1)\ird{|\nabla w|^2}-\beta\ird{|w|^{q+1}}+\mathcal K\,\(\mathsf C_{\rm{GN}}\)^{-(q+1)(1+\eta)}\(\ird{|w|^{2q}}\)^\frac\alpha{2q}
\]
is nonnegative and achieves its minimum at $w_*(x)=(1-|x|^2)_+^\frac1{q-1}$. Hence we have that
\[
\mathsf J[w]\ge\mathsf J[w_*]=0
\]
and this inequality is equivalent to \eqref{GN2}, after an optimization under scaling. Notice that $\vartheta=\frac{2\,\eta}{(q+1)\,(1+\eta)}$.
\begin{thm}\label{Thm:GN2} With the above notations and $\varphi$ given by~\eqref{Eqn:VarphiPM}, we have
\be{Ineq:GNImproved2}
\mathsf J[w]\ge\varphi\left[\beta\(\ird{|w|^{q+1}}-\ird{|w_*|^{q+1}}\)\right]\quad\forall\,w\in\mathrm L^{q+1}(\R^d)\quad\mbox{such that}\quad\ird{|\nabla w|^2}<\infty\,.
\ee
\end{thm}
\begin{proof} The proof is similar to the proof of Theorem~\ref{Thm:GN1} except that the roles of $\ird{|w_*|^{q+1}}$ and $\ird{|w_*|^{2q}}$ are exchanged. \end{proof}

\section{Some consequences for diffusion equations}\label{Sec:Diffusion}

\subsection{Linear case: the Ornstein-Uhlenbeck equation}

Let us consider the Ornstein-Uhlenbeck equation (or backward Kolmogorov equation)
\be{Eqn:OU}
\frac{\partial f}{\partial t}=\Delta f-x\cdot\nabla f
\ee
with initial datum $f_0\in\mathrm L^1_+(\R^d,(1+|x|^2)\,d\mu$ and define the \emph{entropy} as
\[
\mathcal E[f]:=\irdmu{f\,\log f}\,.
\]
Using~\eqref{Ineq:LogSobGaussian}, a standard computation shows that a solution $f=(t,\cdot)$ to~\eqref{Eqn:OU} satisfies
\[
\frac d{dt}\mathcal E[f]=-\,4\irdmu{|\nabla\sqrt f|^2}\le-\,2\,\mathcal E[f]\,,
\]
thus proving that
\be{Estim:StandardOU}
\mathcal E[f(t,\cdot)]\le\mathcal E[f_0]\,e^{-2t}\quad\forall\,t\ge0\,.
\ee
It is well known that $M=\irdmu{f(t,\cdot)}$ does not depend on $t\ge0$. Since the second moment evolves according to
\[
\frac d{dt}\irdmu{f\,|x|^2}=2\irdmu{f\,(d-|x|^2)}\,,
\]
if we assume that $\irdmu{f_0\,|x|^2}=d\,M$, then we get also that $\irdmu{f(t,\cdot)\,|x|^2}=d\,M$ for any $t\ge0$.
\begin{thm}\label{Thm:RateLinear} Let $d\ge 1$ and consider a nonnegative solution to \eqref{Eqn:OU} with initial datum $f_0$ such that $\mathcal E[f_0]$ is finite and $\irdmu{f_0\,|x|^2}=d\,\irdmu{f_0}$. Then we have
\be{Estim:ImprovedOU}
\mathcal E[f(t,\cdot)]\le-\,\frac d2\,\log\left[1-\(1-e^{-\frac 2d\,\mathcal E[f_0]}\)\,e^{-2t}\right]\quad\forall\,t\ge0\,.
\ee
\end{thm}
\begin{proof} The proof relies on the estimate
\[
\frac d{dt}\mathcal E[f]=-\,d\,\irdmu{|\nabla\sqrt f|^2}\ge\frac d4\left[\exp\(\frac 2d\,\mathcal E[f]\)-1\right]\,,
\]
according to~\eqref{Eqn:LogSobGaussianSimple}.\end{proof}
Let us conclude this section on the Ornstein-Uhlenbeck equation with some remarks.\begin{description}
\item[(i)] The estimate~\eqref{Estim:ImprovedOU} is better than \eqref{Estim:StandardOU}: if we let $x=e^{-2t}$ and $a=e^{-\frac 2d\,\mathcal E[f_0]}$, then
\[
-\,\frac d2\,\log\left[1-\(1-e^{-\frac 2d\,\mathcal E[f_0]}\)\,e^{-2t}\right]\le\mathcal E[f_0]\,e^{-2t}
\]
for any $t\ge0$ is equivalent to prove that $h(x)=1-a^x-(1-a)\,x$ is nonnegative for any $a\ge0$ and any $x\in(0,1)$. This is indeed the case because $h(0)=h(1)=0$ and $h''(x)=-\,a^x\,(\log x)^2<0$.
\item The improvement degenerates as $\mathcal E[f_0]\to0_+$. Indeed, we may observe that, for a given $t\ge0$,
\[
-\,\frac d2\,\log\left[1-\(1-e^{-\frac 2d\,\mathcal E_0}\)\,e^{-2t}\right]\sim\mathcal E_0\,e^{-2t}\quad\mbox{as}\quad\mathcal E_0\to0_+\,.
\]
\item[(ii)] Similar results can be obtained using the \emph{generalized Poincar\'e inequalities for Gaussian measures} of Section~\ref{Sec:Beckner}. With $q=2/p$, we get that
\[
\frac d{dt}\irdmu{\frac{f^q-M^q}{q-1}}=-\,\frac4q\irdmu{|\nabla f^{q/2}|^2}\le\Psi_p\(\irdmu{\frac{f^q-M^q}{q-1}}\)
\]
with $\Psi_p(t):=\max\big\{2,\frac{2\,p}{2-p}\,\varphi_p(\frac{2-p}p\,t)\big\}=\frac{2\,p}{2-p}\,\Phi_p(\frac{2-p}p\,t)$ with $\Phi_p$ as in~\eqref{Phi_p}, and thus get an improvement of the standard estimate
\[
\irdmu{\frac{f^q-M^q}{q-1}}\le\irdmu{\frac{f_0^q-M^q}{q-1}}\;e^{-2t}\quad\forall\,t\ge0
\]
if $\irdmu{\frac{f_0^q-M^q}{q-1}}<\frac p{2-p}\,x_\star(p)$.

\item[(iii)] None of the above improvements requires that $\irdmu{f\,x}=0$. If this condition is added at $t=0$, it is preserved by the flow corresponding to \eqref{Eqn:OU} and spectral methods allows to prove that
\[
\irdmu{\frac{f^q-M^q}{q-1}}\le\irdmu{\frac{f_0^q-M^q}{q-1}}\;e^{-2\lambda t}\quad\forall\,t\ge0
\]
where $\lambda$ is the best constant in the inequality
\[
\nrmu u2^2-\nrmu up^2\le\frac1\lambda\,\nrmu{\nabla u}2^2
\]
for any $u\in\mathrm H^1(\R^d,d\mu)$ such that $\irdmu{u\,(1,x,|x|^2-d)}=(0,0,0)$. According to \cite[Theorem~2.4]{MR2375056}, we know that $\lambda\ge\frac{\lambda_3}{1-(p-1)^{\lambda_3/\lambda_1}}$ where $(\lambda_i)_{i\ge0}$ are the eigenvalues of the Ornstein-Uhlenbeck operator $-\Delta+x\cdot\nabla$. Standard results on the harmonic oscillator allow us to prove that $\lambda_i=i$, where the eigenspaces associated with $i=0$, $i=1$ and $i=2$ are generated respectively by the constants, $x_i$ with $i=1$, $2$,... $d$ and $|x|^2-d$.
\end{description}
 
\subsection{Nonlinear case: the fast diffusion equation}

The inequality $\mathsf J[w]\ge0$ in Section~\ref{Sec:GN1} is also known as the entropy -- entropy-production inequality for the fast diffusion equation, as was shown in \cite{MR1940370}. Also see \cite{MR2065020} for a review on these methods. Here are some details. Let us consider the free energy
\[
\mathcal F[v]:=\frac1{p-1}\ird{\(v^p-\mathfrak B^p-\,p\,\mathfrak B^{p-1}\,(v-\mathfrak B)\)}
\]
where the Barenblatt profile $\mathfrak B$ is defined by
\[
\mathfrak B(x)=\(1+|x|^2\)^\frac 1{p-1}\quad\forall\,x\in\R^d\,.
\]
and has mass
\[
M_*:=\ird{\mathfrak B}=\pi^\frac d2\,\frac{\Gamma\left(\frac1{1-p}-\frac d2\right)}{\Gamma\left(\frac1{1-p}\right)}\,.
\]
Next we can define the generalized Fisher information by
\[
\mathcal I[v]:=\frac p{1-p}\ird{v\left|\,\nabla v^{p-1}-2\,x\right|^2}
\]
and consider the deficit functional
\[
\mathcal J[v]:=\mathcal I[v]-4\,\mathcal F[v]\,.
\]
We may also define the temperature as
\[
\Theta[v]:=\frac 1d\,\frac{\ird{|x|^2\,v}}{\ird v}\,.
\]
It turns out that $\mathcal J[v]=\mathsf J[w]$ if $v^{p-\frac12}=w$ and $q=\frac1{2\,p-1}$. Notice that $q\in\big(1,\frac d{d-2}\big)$ is equivalent to $p\in(p_1,1)$ with $p_1:=\frac{d-1}d$. Recall that $\beta=\frac{2\,q}{q-1}-d=\frac1{1-p}-d$. If $\Theta[v]=\Theta[\mathfrak B]$, we observe that $\mathcal F[v]:=\frac1{p-1}\ird{\(v^p-\mathfrak B^p\)}$ because $\ird{\mathfrak B^{p-1}\,(v-\mathfrak B)}=0$. Theorem~\ref{Thm:GN1} can be rephrased in terms of $v$ as follows.
\begin{cor}\label{Cor:GN1} Let $p\in(p_1,1)$. Assume that $v$ is a nonnegative function in $\mathrm L^1(\R^d)$ is such that $v^p$ and $v\,|x|^2$ are both in $\mathrm L^1(\R^d)$, and $\nabla v$ is in $\mathrm L^2(\R^d)$. With the above notations and $\varphi$ defined by \eqref{Eqn:VarphiFD} we have
\[
\mathcal J[v]\ge\varphi\big(\beta\,(1-p)\,\mathcal F[v]\big)\quad\mbox{if}\quad\ird v=M_*\quad\mbox{and}\quad\Theta[v]=\Theta[\mathfrak B]\,.
\]
\end{cor}
Since $\varphi''(s)\ge\varphi''(0)$ for any admissible $s\ge0$, we have in particular that
\[
\mathcal J[v]\ge\kappa\,\big(\mathcal F[v]\big)^2
\]
under the assumptions of Corollary~\ref{Cor:GN1}, with $\kappa=\frac 12\,\beta^2\,(1-p)^2\,\varphi''(0)$. Hence Corollary~\ref{Cor:GN1} allows to recover \cite[Theorem~8]{MR3103175} with a much simpler proof.

\medskip In a second step, we can get rid of the constraints on the mass and on the second moment. Let us define
\[
\Theta_*:=\Theta[\mathfrak B]\,.
\]
If we write that
\[
u(x)=\lambda^{1+\alpha\,d}\,\sigma^{-\frac d2}\,v\(\lambda^\alpha\,x/\sqrt\sigma\)\quad\mbox{with}\quad\lambda=\frac M{M_*}
\]
for some $\alpha\in\R$ to be determined later and choose $\sigma>0$ such that
\[
\sigma=\sigma[u]:=\lambda^{-(1+2\alpha)}\,\frac{\ird{u\,|x|^2}}{\ird{\mathfrak B\,|x|^2}}\,,
\]
then we can define the corresponding \emph{Barenblatt profile} by
\[
\mathfrak B_{M,\sigma}(x)=\lambda^{1+\alpha\,d}\,\sigma^{-\frac d2}\,\mathfrak B\(\lambda^\alpha\,x/\sqrt\sigma\)\quad\forall\,x\in\R^d\,,
\]
and the relative entropy and the relative Fisher information respectively by
\[
\mathcal F_{M,\sigma}[u]:=\frac1{p-1}\ird{\(u^p-\mathfrak B_{M,\sigma}^p-\,p\,\mathfrak B_{M,\sigma}^{p-1}\,(u-\mathfrak B_{M,\sigma})\)}
\]
and, with $p_c:=\frac{d-2}d$,
\[
\mathcal I_{M,\sigma}[u]:=\frac p{1-p}\,\sigma^{\frac d2\,(p-p_c)}\ird{u\left|\,\nabla u^{p-1}-\nabla\mathfrak B_{M,\sigma}^{p-1}\right|^2}\,.
\]
The Barenblatt profile $\mathfrak B_{M,1}(x)$ takes the form $(C_M+|x|^2)^{1/(p-1)}$ for some positive constant $C_M=\lambda^\frac{1+\alpha\,d}{p-1}$ if and only if
\[
\alpha=\frac{1-p}{2-d\,(1-p)}\,.
\]
With this choice, the functionals $\mathcal F_{M,\sigma}[u]$ and $\mathcal I_{M,\sigma}[u]$ have the same scaling properties, so that we get
\[
\mathcal J_{M,\sigma}[u]:=\mathcal I_{M,\sigma}[u]-4\,\mathcal F_{M,\sigma}[u]=\lambda^\frac{2\,p-\,d\,(1-p)}{2-d\,(1-p)}\,\sigma^{\frac d2\,(1-p)}\,\mathcal J[v]\,.
\]
By expressing the values of $\lambda$ and $\sigma$ in terms of $M$ and $\Theta[u]$, we have that
\[
\lambda=\frac M{M_*}\quad\mbox{and}\quad\sigma=\sigma[u]=\(\frac{M_*}M\)^\frac{2\,(2-p)-\,d\,(1-p)}{2-\,d\,(1-p)}\,\frac{\Theta[u]}{\Theta_*}\,,
\]
\[
\lambda^\frac{2\,p-\,d\,(1-p)}{2-d\,(1-p)}\,\sigma^{\frac d2\,(1-p)}=\mathsf h(M,\Theta[u])\quad\mbox{with}\quad\mathsf h(M,\Theta):=\(\frac M{M_*}\)^{\frac{d^2\,(1-p)^2-\,2\,d\,(p^2-4\,p+3)+4\,p}{2\,[2-\,d\,(1-p)]}}\,\(\frac\Theta{\Theta_*}\)^{\frac d2\,(1-p)}
\]
We can now rephrase Corollary~\ref{Cor:GN1} for a general function $u$ as follows.
\begin{cor}\label{Cor:GNUnscaled} Let $p\in(p_1,1)$. Assume that $u$ is a nonnegative function in $\mathrm L^1(\R^d)$ is such that $u^p$ and $u\,|x|^2$ are both in $\mathrm L^1(\R^d)$, and $\nabla u$ is in $\mathrm L^2(\R^d)$. With the same notations as in Corollary~\ref{Cor:GN1} and $\varphi$ defined by \eqref{Eqn:VarphiFD} we have
\[
\mathcal J[u]_{M,\sigma}\ge\mathsf h(M,\Theta[u])\,\varphi\(\beta\,(1-p)\,\frac{\mathcal F_{M,\sigma}[u]}{\mathsf h(M,\Theta[u])}\)\quad\mbox{with}\quad\sigma=\sigma[u]\,.
\]
\end{cor}
The choice $\sigma=\sigma[u]$ is remarkable because
\[
\mathcal F_{M,\sigma[u]}[u]=\inf_{\sigma>0}\mathcal F_{M,\sigma}[u]\,,
\]
so that $\mathfrak B_{M,\sigma[u]}$ is the \emph{best matching Barenblatt profile}, among all Barenblatt profiles with mass $M$ and characteristic scale $\sigma>0$, when measured in relative entropy. See \cite{1004,MR3103175} for more details.

One of the interests of the statement of Corollary~\ref{Cor:GNUnscaled} is that the free energy $\mathcal F_{M,\sigma[u]}[u]$ is an explicit distance to the manifold of the optimal functions for the Gagliardo-Nirenberg inequalities. We have indeed the following Csisz\'ar-Kullback inequality.
\begin{thm}\cite[Theorem~4]{MR3103175}\label{Thm:CK} Let $d\ge 1$, $p\in(d/(d+2),1)$ and assume that $u$ is a non-negative function in $\mathrm L^1(\R^d)$ such that $u^p$ and $x\mapsto |x|^2\,u$ are both integrable on $\R^d$. If $\nrm u1=M$, then
\[
\mathcal F_{M,\sigma[u]}[u]\ge\frac{p\,\sigma[u]^{\frac d2(1-p)}}{8\ird{\mathfrak B_{M,1}^p}}\(C_M\nrm{u-\mathfrak B_{M,\sigma[u]}}1+\frac1{\sigma[u]}\,\ird{|x|^2\,|u-\mathfrak B_{M,\sigma[u]}|}\)^2\,.
\]
\end{thm}

The fast diffusion equation written in self-similar variables is
\be{FDE}
\frac{\partial u}{\partial t}+\nabla\left[v\cdot\(\sigma^{\frac d2\,(p-p_c)}\,\nabla v^{p-1}-2\,x\)\right]=0
\ee
and it is equivalent to the fast diffusion equation $\frac{\partial v}{\partial t}=\Delta v^p$ up to a rescaling: see for instance \cite{MR1940370,BBDGV} when $\sigma$ is taken constant and \cite{1004,MR3103175,DTS} when $\sigma=\sigma[u(t,\cdot)]$ depends on $t$. By using Corollary~\ref{Cor:GNUnscaled}, we get that
\[
\frac d{dt}\mathcal F_{M,\sigma[u(t,\cdot)]}[u(t,\cdot)]=-\,\mathcal I_{M,\sigma[u(t,\cdot)]}[u(t,\cdot)]
\]
while a direct computation shows that
\[
\frac d{dt}\sigma[u(t,\cdot)]=-\,\frac{2\,(1-p)^2}{p\,M\,\Theta[\mathfrak B_{M,1}]}\,\sigma[u(t,\cdot)]^{\frac d2(p-p_c)}\,\mathcal F_{M,\sigma[u(t,\cdot)]}[u(t,\cdot)]\le0\,.
\]
Altogether this establishes a faster convergence rate of the solutions towards the Barenblatt profiles than one would get using the entropy -- entropy production inequality $\mathcal J_{M,\sigma[u]}[u]\ge0$, and even a better rate that the one found in \cite{MR3103175,DTS} because $\varphi''(s)>\varphi''(0)$. Details are out of the scope of this paper and a simplified method will appear in \cite{DTS}.

\subsection{Nonlinear case: the porous medium equation}

Now we turn our attention to the inequality $\mathsf J[w]\ge0$ in Section~\ref{Sec:GN2} which was also studied in \cite{MR1940370}. With $p>1$, we may consider the free energy
\[
\mathcal F[v]:=\frac1{p-1}\ird{\(v^p-\mathfrak B^p+\,p\,|x|^2\,(v-\mathfrak B)\)}
\]
for any nonnegative function $v$ such that $\ird v=\ird{\mathfrak B}$, where the Barenblatt profile $\mathfrak B$ is now defined by
\[
\mathfrak B(x)=\(1-|x|^2\)_+^\frac 1{p-1}\quad\forall\,x\in\R^d\,.
\]
and has mass
\[
M_*:=\pi^\frac d2\,\frac{\Gamma\left(\frac p{p-1}\right)}{\Gamma\left(\frac p{p-1}+\frac d2\right)}\,.
\]
Next we can define the generalized Fisher information by
\[
\mathcal I[v]:=\frac p{p-1}\ird{v\left|\,\nabla v^{p-1}-2\,x\right|^2}
\]
and consider the deficit functional
\[
\mathcal J[v]:=\mathcal I[v]-4\,\mathcal F[v]\,.
\]
As before, we may also define the temperature as
\[
\Theta[v]:=\frac 1d\,\frac{\ird{|x|^2\,v}}{\ird v}\,.
\]
It turns out that $\mathcal J[v]=\mathsf J[w]$ if $v^{p-\frac12}=w$ and $q=\frac1{2\,p-1}$. Notice that $q\in(0,1)$ is equivalent to $p\in(1,\infty)$. Recall that $\beta=\frac{2\,q}{1-q}+d=\frac1{p-1}+d$. If $\Theta[v]=\Theta[\mathfrak B]$, we observe that $\mathcal F[v]:=\frac1{p-1}\ird{\(v^p-\mathfrak B^p\)}$. Theorem~\ref{Thm:GN2} can be rephrased in terms of $v$ as follows.
\begin{cor}\label{Cor:GN2} Let $p\in(1,+\infty)$. Assume that $v$ is a nonnegative function in $\mathrm L^1(\R^d)$ is such that $v^p$ and $v\,|x|^2$ are both in $\mathrm L^1(\R^d)$, and $\nabla v$ is in $\mathrm L^2(\R^d)$. With the above notations and $\varphi$ defined by \eqref{Eqn:VarphiPM} we have
\[
\mathcal J[v]\ge\varphi\big(\beta\,(p-1)\,\mathcal F[v]\big)\quad\mbox{if}\quad\ird v=M_*\quad\mbox{and}\quad\Theta[v]=\Theta[\mathfrak B]\,.
\]
\end{cor}
Since $\varphi''(s)\ge\varphi''(0)$ for any admissible $s\ge0$, we have in particular that
\[
\mathcal J[v]\ge\kappa\,\big(\mathcal F[v]\big)^2
\]
under the assumptions of Corollary~\ref{Cor:GN2}, with $\kappa=\frac 12\,\beta^2\,(1-p)^2\,\varphi''(0)$. The result of Corollary~\ref{Cor:GN2} is new.

\medskip In a second step, we can get rid of the constraints on the mass and on the second moment. Let us define
\[
\Theta_*:=\Theta[\mathfrak B]\,.
\]
If we write that
\be{uv:rescaling}
u(x)=\lambda^{1+\alpha\,d}\,\sigma^{-\frac d2}\,v\(\lambda^\alpha\,x/\sqrt\sigma\)\quad\mbox{with}\quad\lambda=\frac M{M_*}
\ee
for some $\alpha\in\R$ to be determined later and choose $\sigma>0$ such that
\[
\sigma=\sigma[u]:=\lambda^{-(1+2\alpha)}\,\frac{\ird{u\,|x|^2}}{\ird{\mathfrak B\,|x|^2}}\,,
\]
then we can define the corresponding \emph{Barenblatt profile} by
\[
\mathfrak B_{M,\sigma}(x)=\lambda^{1+\alpha\,d}\,\sigma^{-\frac d2}\,\mathfrak B\(\lambda^\alpha\,x/\sqrt\sigma\)\quad\forall\,x\in\R^d\,,
\]
and the relative entropy and the relative Fisher information respectively by
\[
\mathcal F_{M,\sigma}[u]:=\frac1{p-1}\ird{\(u^p-\mathfrak B_{M,\sigma}^p+\,p\,\sigma^{-\frac d2\,(p-p_c)}\,|x|^2\,(u-\mathfrak B_{M,\sigma})\)}
\]
and
\[
\mathcal I_{M,\sigma}[u]:=\frac p{p-1}\,\sigma^{\frac d2\,(p-p_c)}\ird{u\left|\,\nabla u^{p-1}-\,2\,x\,\sigma^{-\frac d2\,(p-p_c)}\right|^2}\,.
\]
where $p_c=\frac{d-2}d$. The Barenblatt profile $\mathfrak B_{M,1}(x)$ takes the form $(C_M-|x|^2)_+^{1/(p-1)}$ for some positive constant $C_M=\lambda^\frac{1+\alpha\,d}{p-1}$ if and only if
\[
\alpha=-\,\frac{p-1}{2+d\,(p-1)}\,.
\]
With this choice, the functionals $\mathcal F_{M,\sigma}[u]$ and $\mathcal I_{M,\sigma}[u]$ have the same scaling properties, so that we get
\[
\mathcal J_{M,\sigma}[u]:=\mathcal I_{M,\sigma}[u]-4\,\mathcal F_{M,\sigma}[u]=\lambda^\frac{2\,p+\,d\,(p-1)}{2+d\,(p-1)}\,\sigma^{\frac d2\,(1-p)}\,\mathcal J[v]\,.
\]
By expressing the values of $\lambda$ and $\sigma$ in terms of $M$ and $\Theta[u]$, we have that
\[
\lambda=\frac M{M_*}\quad\mbox{and}\quad\sigma=\sigma[u]=\(\frac{M_*}M\)^\frac{2\,(2-p)-\,d\,(1-p)}{2-\,d\,(1-p)}\,\frac{\Theta[u]}{\Theta_*}\,,
\]
\[
\lambda^\frac{2\,p-\,d\,(1-p)}{2-d\,(1-p)}\,\sigma^{\frac d2\,(1-p)}=\mathsf h(M,\Theta[u])\quad\mbox{with}\quad\mathsf h(M,\Theta):=\(\frac M{M_*}\)^{\frac{d^2\,(1-p)^2-\,2\,d\,(p^2-4\,p+3)+4\,p}{2\,[2-\,d\,(1-p)]}}\,\(\frac\Theta{\Theta_*}\)^{\frac d2\,(1-p)}
\]
We can now rephrase Corollary~\ref{Cor:GN2} for a general function $u$ as follows.
\begin{cor}\label{Cor:GNUnscaled2} Let $p\in(1,+\infty)$. Assume that $u$ is a nonnegative function in $\mathrm L^1(\R^d)$ is such that $u^p$ and $u\,|x|^2$ are both in $\mathrm L^1(\R^d)$, and $\nabla u$ is in $\mathrm L^2(\R^d)$. With the same notations as in Corollary~\ref{Cor:GN2} and $\varphi$ defined by \eqref{Eqn:VarphiPM} we have
\[
\mathcal J[u]_{M,\sigma}\ge\mathsf h(M,\Theta[u])\,\varphi\(\beta\,(p-1)\,\frac{\mathcal F_{M,\sigma}[u]}{\mathsf h(M,\Theta[u])}\)\quad\mbox{with}\quad\sigma=\sigma[u]\,.
\]
\end{cor}
For the same reason as in the fast diffusion case, the choice $\sigma=\sigma[u]$ is remarkable because
\[
\mathcal F_{M,\sigma[u]}[u]=\inf_{\sigma>0}\mathcal F_{M,\sigma}[u]\,,
\]
so that $\mathfrak B_{M,\sigma[u]}$ is the \emph{best matching Barenblatt profile}, among all Barenblatt profiles with mass $M$ and characteristic scale $\sigma>0$, when measured in relative entropy.

One of the interests of the statement of Corollary~\ref{Cor:GNUnscaled2} is that the free energy $\mathcal F_{M,\sigma[u]}[u]$ is an explicit distance to the manifold of the optimal functions for the Gagliardo-Nirenberg inequalities. Variants of the Pinsker-Csisz\'ar-Kullback inequality can be found in \cite{MR1777035,MR1940370} and allow to control $\nrm{u-\mathfrak B_{M,\sigma[u]}}1$ under additional assumptions. Here we shall give a simpler result which goes as follows.
\begin{thm}\label{Thm:CK2} Let $d\ge 1$, $p\in(1,\infty)$ and assume that $u$ is a non-negative function in $\mathrm L^1(\R^d)$ such that $u^p$ and $x\mapsto |x|^2\,u$ are both integrable on $\R^d$. With previous notations, we have that
\[
\mathcal F_{M,\sigma[u]}[u]\ge\frac p{p-1}\,\min\{1,p-1\}\,\nrm{u-\mathfrak B_{M,\sigma[u]}}p^p\,.
\]
\end{thm}
\begin{proof} Let us consider the function
\[
\chi(t):=t^p-1-p\,(t-1)-c_p\,|t-1|^p
\]
with $c_p:=\min\{1,p-1\}$. The reader is invited to check that for some $t_p\in(0,1)$, $\chi''(t)<0$ for any $t\in(0,t_p)$, $\chi''(t)>0$ for any $t\in(t_p,1)\cup(1,+\infty)$. Since $\chi(0)=\chi(1)=0$ and $\chi'(0)=0$, we get that $\chi(t)\ge0$ for any $t\ge0$. As above, let us define $v$ by the rescaling~\eqref{uv:rescaling} and consider
\[
\mathcal F[v]=\frac 1{p-1}\int_{|x|\le1}\chi\(\frac v{\mathfrak B}\)\mathfrak B^p\;dx+\frac 1{p-1}\int_{|x|>1}\(u^p-p\,(1-|x|^2)\,u\)\;dx\,.
\]
It is straightforward to check that
\[
\mathcal F[v]=\frac{c_p}{p-1}\int_{|x|\le1}\left|v-\mathfrak B\right|^p\;dx+\frac 1{p-1}\int_{|x|>1}u^p\;dx\ge c_p\,\nrm{v-\mathfrak B}p^p
\]
and we conclude using the identities
\[
\mathcal F_{M,\sigma[u]}[u]=\mathsf h(M,\Theta[u])\,\mathcal F[v]\quad\mbox{and}\quad\nrm{u-\mathfrak B_{M,\sigma[u]}}p^p=\mathsf h(M,\Theta[u])\,\nrm{v-\mathfrak B}p^p\,.
\]
\end{proof}

The porous medium equation written in self-similar variables is
\be{PME}
\frac{\partial u}{\partial t}=\nabla\left[v\cdot\(\sigma^{\frac d2\,(p-p_c)}\,\nabla v^{p-1}+2\,x\)\right]=0
\ee
and it is equivalent to the porous medium equation $\frac{\partial v}{\partial t}=\Delta v^p$ up to a rescaling: as in \cite{1004,MR3103175,DTS} we may choose $\sigma=\sigma[u(t,\cdot)]$ to depend on $t$. By using Corollary~\ref{Cor:GNUnscaled2}, we get that
\[
\frac d{dt}\mathcal F_{M,\sigma[u(t,\cdot)]}[u(t,\cdot)]=-\,\mathcal I_{M,\sigma[u(t,\cdot)]}[u(t,\cdot)]
\]
while a direct computation shows that
\[
\frac d{dt}\sigma[u(t,\cdot)]=-\,\frac{2\,(p-1)^2}{p\,M\,\Theta[\mathfrak B_{M,1}]}\,\sigma[u(t,\cdot)]^{\frac d2(p-p_c)}\,\mathcal F_{M,\sigma[u(t,\cdot)]}[u(t,\cdot)]\le0\,.
\]
Altogether this establishes a faster convergence rate of the solutions towards the Barenblatt profiles than one would get using the entropy -- entropy production inequality $\mathcal J_{M,\sigma[u]}[u]\ge0$. This is new. More details based on a simpler method will appear in \cite{DTS}.

\medskip As a concluding remark, which is valid for the porous medium case, for the fast diffusion case, and also for the linear case of the Ornstein-Uhlenbeck equation, we may observe that two moments are involved while the center of mass $\ird{x\,u}$ is not supposed to be equal to $0$. In that sense our measurement of the distance to the manifold of optimal functions by the deficit functional is an explicit but still a rough estimate, which is definitely not optimal at least in the perturbation regime, that is, close to the optimal functions, which is precisely the regime that was considered in \cite{MR1124290}.

\par\medskip\centerline{\rule{2cm}{0.2mm}}\medskip
\begin{spacing}{0.9}
\noindent{\small{\bf Acknowlegments.} This work has been partially supported by the projects \emph{STAB}, \emph{NoNAP} and \emph{Kibord} of the French National Research Agency (ANR). J.D.~thanks the Department of Mathematics of the University of Pavia for inviting him and G.~Savar\'e for stimulating discussions. The authors thank F.~Bolley for pointing them a missing reference.
\par\medskip\noindent{\small\copyright\,2014 by the authors. This paper may be reproduced, in its entirety, for non-commercial purposes.}}\end{spacing}

\end{document}